\documentclass[11pt]{amsart}

\usepackage{amscd,amsmath,amssymb,fancyhdr}
\usepackage{stackengine}
\usepackage{scalerel}
\usepackage{mathtools,bm}
\usepackage{tikz-cd}

\newcommand{\al}{\alpha}

\newcommand{\f}{\varphi}

\newcommand{\Ker}{\text{Ker}}

\stackMath
\newcommand\reallywidehat[1]{%
\savestack{\tmpbox}{\stretchto{%
  \scaleto{%
    \scalerel*[\widthof{\ensuremath{#1}}]{\kern-.6pt\bigwedge\kern-.6pt}%
    {\rule[-\textheight/2]{1ex}{\textheight}}
  }{\textheight}%
}{0.5ex}}%
\stackon[1pt]{#1}{\tmpbox}%
}
\newcommand{\RR}{\mathbb{R}}

\newcommand{\SL}{\mathrm{SL}}

\numberwithin{equation}{section}

\def\eqref#1{(\ref{#1})}

\newcommand{\Z}{{\mathbb Z}}
\newcommand{\C}{{\mathbb C}}
\newcommand{\R}{{\mathbb R}}

\renewcommand{\H}{{\mathbb H}}

\def\1{\sqrt{-1}\:}

\newcommand{\cntrct}                
{\hspace{2pt}\raisebox{1pt}{\text{$\lrcorner$}}\hspace{2pt}}
\newcommand{\arrow}{{\:\longrightarrow\:}}

\renewcommand{\bar}{\overline}
\renewcommand{\phi}{\varphi}
\renewcommand{\epsilon}{\varepsilon}
\renewcommand{\geq}{\geqslant}
\renewcommand{\leq}{\leqslant}

\renewcommand{\span}{{\rm span}}

\renewcommand{\dim}{\operatorname{dim}}

\renewcommand{\Re}{\operatorname{Re}}
\renewcommand{\Im}{\operatorname{Im}}

\newcommand{\vol}{\operatorname{vol}}

\newcommand{\ie}{{\em i.e. }}
\newcommand{\eg}{{\em e.g. }}

\newcounter{Mycounter}[section]
\newcounter{lemma}[section]
\newcounter{claim}[section]

\newcounter{sublemma}[section]

\newcounter{corollary}[section]

\newcounter{theorem}[section]

\newcounter{conjecture}[section]

\newcounter{proposition}[section]

\newcounter{definition}[section]

\newcounter{example}[section]

\newcounter{remark}[section]

\newcounter{problem}[section]

\newcounter{question}[section]

\makeatletter

\@addtoreset{equation}{section}

\@addtoreset{footnote}{section}

\makeatother

\def\blacksquare{\hbox{\vrule width 5pt height 5pt depth 0pt}}

\def\endproof{\blacksquare}

\begin{document}

\newpage

\title[Morse-Novikov cohomology of LCK surfaces]{Morse-Novikov cohomology of locally conformally K\" ahler surfaces}
\author{Alexandra Otiman}
\thanks{Partially supported by an Erasmus+ fellowship from University of Bucharest.}

\date{\today}
\address{Institute of Mathematics ``Simion Stoilow'' of the Romanian Academy\\
21, Calea Grivitei Street, 010702, Bucharest, Romania {\em and}\newline University of Bucharest, Faculty of Mathematics and Computer Science, 14 Academiei Str., Bucharest, Romania.}
\email{alexandra\_otiman@yahoo.com}

\abstract  We review the properties of the Morse-Novikov cohomology and compute it for all known  compact complex surfaces with locally conformally K\"ahler metrics. We present explicit computations for the Inoue surfaces $\mathcal{S}^0$, $\mathcal{S}^+$, $\mathcal{S}^-$ and classify the locally conformally K\" ahler (and the tamed locally conformally symplectic) forms on $\mathcal{S}^0$. We prove the nonexistence of LCK metrics with potential and more generally, of $d_\theta$-exact LCK metrics on Inoue surfaces and Oeljeklaus-Toma manifolds.\\[.1in]

\endabstract

\maketitle

\tableofcontents

\section{Introduction}

The {\em Morse-Novikov cohomology} of a manifold $M$ refers to the cohomology of the complex of smooth real forms $\Omega^{\bullet}(M)$, with the differential operator perturbed with a closed one-form $\eta$, defined as follows
\begin{equation}
d_{\eta} := d - \eta \wedge \cdot
\end{equation}
Indeed, the closedness of $\eta$ implies $d_{\eta}^2=0$, whence $d_{\eta}$ produces a cohomology, which we denote by $H^{\bullet}_\eta(M)$. 

Throughout this paper, we shall use the name Morse-Novikov for the cohomology $H^{\bullet}_\eta(M)$, although the name Lichnerowicz cohomology is also used in the literature (see \cite{bk}, \cite{hr}). Its study began with Novikov (\cite{n1}, \cite{n2}) and was independently developed by Guedira and Lichnerowicz (\cite{gl}). 

The Morse-Novikov cohomology has more than one description.
To begin with, consider the following exact sequence of sheafs:
\begin{equation}\label{rezolutie}
0 \rightarrow \Ker\, d_\eta \xrightarrow{i}  \Omega^0_{M}( \cdot ) \xrightarrow{d_{\eta}} \Omega^1_{M}( \cdot )\xrightarrow{d_\eta}\Omega^2_{M}( \cdot )\xrightarrow{d_\eta} \cdots
\end{equation}
where we denote by $\Omega^k_{M}( \cdot )$ the sheaf of smooth real $k$-forms on $M$. In fact, the sequence above is an acyclic resolution for $\Ker\, d_\eta$, as each $ \Omega^i_{M}( \cdot )$ is soft,  \cite[Proposition 2.1.6 and Theorem 2.1.9 ]{d}. Thus, by taking global sections in \eqref{rezolutie}, we compute the cohomology groups of $M$ with values in the sheaf $\Ker\, d_\eta$, $H^i(M, \Ker\, d_\eta)$. What we obtain is actually the Morse-Novikov cohomology.

The sheaf $\Ker\, d_\eta$ has the property that there exists a covering $\{U_i\}_i$ of $M$, such that it is constant when restricted to each $U_i$. In order to see this, one simply takes a contractible covering $\{U_i\}_i$ for which $\eta=df_i|_{U_i}$, then by considering the map $g \mapsto e^{-f_i}g$, one gets an isomorphism $\Ker\, d_{\eta}|_{U_i} \simeq \R$.

Moreover, the covering $\{U_i\}_i$ and the isomorphisms above associate to $\Ker\, d_\eta$ a line bundle $L_\eta$, which is trivial on this covering and whose transition maps are $g_{ij}=e^{f_i -f_j}$. It is immediate that $(U_i, e^{-f_i})$ defines a global nowhere vanishing section $s$ of $L_\eta^*$, which is the dual of $L_\eta$ and by means of $s$, $L_\eta^*$ is isomorphic to the trivial bundle. We define a flat connection $\nabla$ on $L_\eta^*$ by $\nabla s= - \eta \otimes s$. Then $H^i_\eta(M)$ can also be computed as the cohomology of the following complex of forms with values in $L_\eta^*$: 
\begin{equation}
0  \rightarrow  \Omega^0(M, L_\eta^* ) \xrightarrow{\nabla}  \Omega^1(M, L_\eta^* )\xrightarrow{\nabla}  \Omega^2(M, L_\eta^* ) \xrightarrow{\nabla} \cdots
\end{equation}

\begin{remark}\label{spanioli}
Unlike de Rham cohomology,  Morse-Novikov cohomology $H^i_{\eta}$,  is not a topological invariant, it depends on $[\eta] \in H^1_{dR}$. Also, Riemannian properties involving this one-form can be important. For instance, it was shown in \cite{llmp} that if on a compact manifold $M$ there exists a Riemannian metric $g$ and a closed one-form $\eta$ such that $\eta$ is parallel with respect to $g$, then for any $i \geq 0$, $H^i_\eta(M)=0$.
\end{remark}

Some properties verified by the Morse-Novikov cohomology, important for this paper, are summarized in the following:
\begin{proposition} \label{prop1} Let $M$ be a $n$-dimensional manifold and $\eta$ a closed one-form. Then:
\begin{enumerate} 
\item if $\eta'= \eta+ df$, for any $i \geq 0$, $H^i_{\eta'}(M) \simeq H^i_{\eta}(M)$ and the isomorphism is given by the map $[\alpha] \mapsto [e^{-f}\alpha]$.
\item (\cite{hr}, \cite{gl}) if $\eta$ is not exact and $M$ is connected and orientable, $H^0_\eta(M)$ and $H^n_\eta(M)$ vanish.
\item (\cite{bk}) the Euler characteristic of the Morse-Novikov cohomology coincides with the Euler characteristic of the manifold, as a consequence of the Atyiah-Singer index theorem, which implies that the index of the elliptic complex $(\Omega^k(M), d_\eta)$ is independent of $\eta$.
\end{enumerate}
\end{proposition}

Motivated by the natural setting that locally conformally symplectic and locally conformally K\"ahler manifolds provide for the Morse-Novikov cohomology, the aim of this paper is to present some explicit examples and computations on the Inoue surfaces $\mathcal{S}^0$, $\mathcal{S}^+$ and $\mathcal{S}^-$. Moreover, regarding the recent results in  \cite{ad}, we also draw some consequences involving the locally conformally K\"ahler metrics or more general, tamed  locally conformally symplectic forms the on the surface $\mathcal{S}^0$, and we prove  that the Inoue surfaces cannot bear LCK metrics with potential.

The paper is organized as follows. Section 2 is devoted to introducing locally conformally symplectic and locally conformally K\"ahler manifolds. In Section 3, we compute the Morse-Novikov cohomology of the Inoue surface $\mathcal{S}^0$ and classify the locally conformally K\"ahler metrics on  $\mathcal{S}^0$. In Section 4 we consider the Inoue surfaces $\mathcal{S}^+$ and $\mathcal{S}^-$, and prove the nonexistence of LCK metrics either with potential or with $d_\theta$-exact fundamental form on all the Inoue surfaces and on Olejeklaus-Toma manifolds.  In Section 5 we give a brief overview  of  the Morse-Novikov cohomology of LCK surfaces in class $\mathrm{VI}$ and $\mathrm{VII}$.

\section {Locally conformally symplectic  and locally conformally K\"ahler  manifolds}

\begin{definition} {\em Locally conformally symplectic manifolds} (shortly LCS) are smooth real (necessarily even-dimensional) manifolds endowed with a nondegenerate two-form $\omega$ which satisfies the equality
\begin{equation} \label{lcs} 
d \omega = \theta \wedge \omega
\end{equation}
for some closed one-form $\theta$, called the {\em Lee form}. 
\end{definition}

Equivalently, this means there exists a non-degenerate two-form $\omega$, an open  covering $\{U_i\}_i$ of the manifold,  and smooth functions $f_i$ on $U_i$ such that $e^{-f_i}\omega$ are symplectic, which literally explains their name. 

Equality \eqref{lcs} rewrites as $d_\theta \omega=0$, hence the problem of studying on an LCS manifold the Morse-Novikov cohomology associated to the Lee form of an LCS structure is natural.  

\begin{definition} On a complex manifold $X=(M, J)$, a Hermitian metric $g$ is called {\em locally conformally K\"ahler} (shortly LCK) if there exists a closed one-form $\theta$ such that the fundamental two-form $\omega$ associated to $g$ satisfies $d\omega=\theta \wedge \omega$.
\end{definition}

Equivalently, a (complex) manifold $M$ ($(M,J)$) is LCS (LCK) if it admits a symplectic (K\"ahler) cover $(\tilde M,\Omega)$ such that the deck group acts by homotheties with respect to the symplectic (K\"ahler) form $\Omega$.  The pull-back $\tilde\omega$ of the LCS (LCK) form of $M$ is then conformal to the symplectic (K\"ahler) form of the covering. 

There are many examples of LCS manifolds coming from LCK geometry (see \cite{do}), but the relation between LCS and LCK is interesting and similar to the relation between symplectic and K\"ahler manifolds. Examples of LCS manifolds which are not LCK were constructed in \cite{bk2} and \cite{bm}.
Examples of LCK manifolds include the Hopf manifolds $S^1 \times S^{2n-1}$, the Inoue surfaces $\mathcal{S}^0$, $\mathcal{S}^-$ and some wide  subclasses of the Inoue surface $\mathcal{S}^+$ (see \cite{t}) and some higher dimensional analogues of $\mathcal{S}^0$ called Oeljeklaus-Toma manifolds (see \cite{ot}).

Among the LCK metrics, the following two are of special interest and were intensively studied:

\begin{definition} An LCK metric $g$ on a complex manifold $(X, J)$ is called {\em Vaisman} if the fundamental two-form $\omega$ of $g$ is parallel with respect to the Levi-Civita connection of $g$.
\end{definition} 

The prototype of Vaisman manifolds is $S^1 \times S^{2n-1}$, but there are compact LCK manifolds which do not admit Vaisman metrics, such as the LCK Inoue surfaces (see \cite{b}). Since the Lee form is parallel for Vaisman manifolds, the result in \cite{llmp} (see \ref{spanioli}) applies and the Morse-Novikov cohomology with respect to this form vanishes.

\begin{definition} An {\em LCK metric with potential} $g$ on a manifold $X=(M, J)$ is an LCK metric such that there exists a covering $\tilde X$ on which the pull-back $\tilde{\omega}$ of its fundamental form $\omega$ satisfies $\tilde{\omega}=\frac{dd^c f}{f}$, where $d^c=JdJ^{-1}$, for a  plurisubharmonic function $f:\tilde{X}\rightarrow \RR^+$, such that  $\gamma^*f=e^{c_\gamma}f$ ($c_\gamma \in \R$), for any deck transformation $\gamma \in \pi_1(X)$. 

In other words, the K\"ahler metric on the cover $\tilde X$ has global, positive and automorphic potential, see \cite{ov}.
\end{definition}

Vaisman manifolds and non-diagonal Hopf manifolds provide examples of LCK manifolds with potential.  For more details, see \eg \cite{ov}.

Until F. Belgun showed in \cite{b} that there exists no LCK metric on a subclass of Inoue surfaces $\mathcal{S}^+$, it was generally believed that all  complex surfaces with odd first Betti number carry an LCK metric. In this context, the characterization of LCK metrics on complex surfaces is of particular interest.  A weaker condition than LCK was considered in \cite{ad}, namely {\em locally conformally symplectic forms which tame the complex structure} $J$ (this parallels symplectic forms taming a complex structure): 

\begin{definition}(\cite{ad}) A {\em  locally conformally symplectic form} $\omega$ on a complex surface $X=(M, J)$ tames $J$ if $\omega(X, JX) > 0$ for any non-zero vector field $X$ on $M$.
\end{definition}

It was proved in \cite{ad} that any compact complex surface $X=(M, J)$ with odd first Betti number admits a locally conformally symplectic form which tames $J$.  Moreover, in the same paper, the following subsets of $H^1_{dR}(M)$ are introduced:
$$\mathcal{C}(X)=\{[\theta] \in H^1_{dR}(M) \mid \text{there exists}\,\, \omega \in \Omega^{1,1}(X), \omega >0, d_\theta \omega = 0\}$$
$$\mathcal{T}(X)=\{[\theta] \in H^1_{dR}(M) \mid \text{there exists}\,\, \omega \in \Omega^{2}(X), \omega^{1, 1} >0, d_\theta \omega = 0\}$$
as cohomological invariants similar to the  K\" ahler cone in the K\" ahler setting. For the Inoue surfaces $\mathcal{S}^+$ and $\mathcal{S}^-$, the authors characterize the above sets. We here give similar characterizations for $S^0$.

\section{Morse-Novikov cohomology of the Inoue surface $\mathcal{S}^0$}

\subsection{Description of the LCS manifold $S^0$.}\label{descriere} In  \cite{in}, M. Inoue introduced three types of complex compact surfaces, which are traditionally referred to as the Inoue surfaces $\mathcal{S}^0$, $\mathcal{S}^+$ and $\mathcal{S}^-$. In \cite{t}, Tricerri endowed the Inoue surfaces $\mathcal{S}^0$, $\mathcal{S}^-$ and some subclasses of $S^+$ with locally conformally K\" ahler metrics, in particular, by forgetting the complex structure,  with locally conformally symplectic structures. 

We review the construction of $\mathcal{S}^0$ and insist on its description as mapping torus of the 3-dimensional torus $\mathbb{T}^3$.

Let $A$ be a matrix from $\SL_3(\mathbb{Z})$ with one real eigenvalue $\alpha >1$ and two complex eigenvalues $\beta$ and $\bar{\beta}$. We denote by $(a_1,a_2,a_3)^t$  
 a real eigenvector of $\alpha$ and by $(b_1, b_2, b_3)^t$  
 a complex eigenvector of $\beta$. Let $G_A$ be the group of affine transformations of $\C \times \H$ generated by the transformations:
\begin{equation*}
\begin{split}
(z, w)& \mapsto (\beta z, \alpha w),\\
(z, w)& \mapsto (z+b_i, w+a_i).
\end{split}
\end{equation*}
for all $i=1, 2, 3$, where $\H$ stands for the Poincar\'e half-plane.

As a complex manifold,  $\mathcal{S}^0$ is  $(\C \times \H)/G_A$, where the complex structure, which we shall denote by $J$, is the the one inherited from $\C \times \H$. 

We now explain  its structure as a mapping torus. Denote by $\mathbb{T}^3$ the standard 3-dimensional torus, namely $\mathbb{T}^3 = \R^3/ \langle f_1, f_2, f_3 \rangle$, where 
$f_1$ (resp.$f_2$, $f_3$) is the translation with $(1,0,0)$ (resp. $(0,1,0)$, $(0,0,1)$). 

Let $\bm{x}:=(x,y,z)^t$, and  consider the automorphism $\phi$ of $\R^3$ with matrix $A^t$ in the canonical basis.

It clearly descends to an automorphism $\phi$ of $\mathbb{T}^3$, since $A^t$ belongs to $\SL_3(\Z)$. Let $\widehat{\bm{x}}$ denote a point of $\mathbb{T}^3$. We define the manifold
$$\mathbb{T}^3 \times_{\phi}\R^+ : = (\mathbb{T}^3 \times \R^+)/(\widehat{ \bm{x}},t)\sim (\phi(\widehat{\bm{x}}),\al t)$$ %
which has the structure of a compact fiber bundle over $S^1$ by considering 
$$p : \mathbb{T}^3 \times_{\phi}\R^+ \rightarrow S^1,\qquad [(\widehat{ \bm{x}},t)]\mapsto e^{2\pi \mathrm{i} \mathrm{log}_{\alpha}t}$$

Here we denote by $[ , ]$ the equivalence class with respect to $\sim$.

In order to write explicitly a diffeomorphism between $\mathbb{T}^3 \times_{\phi}\R^+$ and $\mathcal{S}^0$, let
$$
B:=\left(\begin{matrix}
    \Re b_1 & \Re b_2 & \Re b_3 \\
    \Im b_1 & \Im b_2 & \Im b_3\\
    a_1 & a_2 & a_3
     \end{matrix} \right)
$$
Now the requested diffeomorphism acts as:
$$[\widehat{\bm{x}},t]\mapsto [[\widehat{B\cdot\bm{x}},t]],$$
where $x+ \mathrm{i}y$ and $z+\mathrm{i}t$ are coordinates on $\C \times \H$ and $\left[\left[x+\mathrm{i}y, z+\mathrm{i}t \right]\right]$ denotes the equivalence class of $\left(x+\mathrm{i}y, z+\mathrm{i}t\right)$, under the action of $G_A$. It is straightforward to check this map is well defined and indeed an isomorphism. 

The LCK structure given by Tricerri in \cite{t} is given as a $G_A$-invariant globally conformally K\"ahler structure on $\C \times \H$ and in the coordinates $(z, w)$, the expressions for the metric and the Lee form, respectively, are:
\begin{equation*}
\begin{split}
g&=-\mathrm{i}\left(\frac{dw \otimes d\overline{w}}{w_2^2}+w_2dz \otimes d\bar{z}\right)\\
\theta& = \frac{dw_2}{w_2},
\end{split}
\end{equation*}
where $w_2=\Im(w)$. For our description as fiber bundle and real coordinates $(x, y, z, t)$, the Lee form $\theta$ is $\frac{dt}{t}$. 

We denote by $\vartheta$ the volume form of the circle of length 1.

A simple computation shows that 
$$\theta=\mathrm{ln}\, {\alpha} \cdot \vartheta.$$

\subsection{Explicit computation of the Morse-Novikov cohomology.}\label{calcul} To compute by hand the Morse-Novikov cohomology groups of $\mathcal{S}^0$, we shall use the following twisted version of the Mayer-Vietoris sequence:

\begin{lemma} (\cite[Lemma 1.2]{hr}) Let $M$ be the union of two open sets $U$ and $V$ and $\theta$ a closed one-form. Then there exists a long exact sequence
\begin{multline}
\cdots \rightarrow H^i_\theta(M)\xrightarrow{\alpha_*}H^i_{\theta_{|U}}(U) \oplus H^i_{\theta_{|V}}(V)\xrightarrow{\beta_*}\\
\xrightarrow{\beta_*}H^i_{\theta_{|U \cap V}}(U \cap V)\xrightarrow{\delta} H^{i+1}_{\theta}(M)\rightarrow \cdots
\end{multline}
where for some  partition of unity $\{\lambda_U, \lambda_V\}$ subordinated to the covering $\{U,V\}$, the above morphisms are:
\begin{equation*}
\begin{split}
\delta ([\sigma]) &= [d\lambda_U \wedge \sigma]=-[d\lambda_V \wedge \sigma],\\
\alpha (\sigma)&= (\sigma_{|U}, \sigma_{|V}),\\
\beta(\sigma, \tau)&=\sigma_{|U \cap V}-\tau_{|U \cap V}.
\end{split}
\end{equation*}
\end{lemma}

We first choose the open sets $U_1$ and $U_2$ which cover the circle:
\begin{equation*}
U_1 : =\{e^{2\pi\mathrm{i}t}\mid t \in (0, 1)\}, \qquad U_2:= \{e^{2\pi\mathrm{i}t}\mid t \in (\tfrac 12,\tfrac{3}{2} )\},
\end{equation*}
and take as open sets $U:= p^{-1}(U_1)$ and $V:= p^{-1}(U_2)$, representing a covering of  $\mathcal{S}^0$. The sets $U$ and $V$ are the trivializations of  $\mathcal{S}^0$ as fiber bundle over $S^1$.  Therefore, we have
\begin{equation*}
\begin{split}
\f_{U_1}&:U{\longrightarrow}U_1\times\mathbb{T}^3,\qquad [\widehat{\bm{x}}, t] \mapsto (e^{2\pi \mathrm{i}\mathrm{log}_{\alpha}t}, \widehat{\bm{x}}),\, t\in(1,\al),\\
\f_{U_2}&:V{\longrightarrow}U_2\times\mathbb{T}^3,\qquad [\widehat{\bm{x}}, t] \mapsto (e^{2\pi \mathrm{i}\mathrm{log}_{\alpha}t}, \widehat{\bm{x}}),\, t\in(\alpha^{\frac{1}{2}}, \alpha^{\frac{3}{2}}).
\end{split}
\end{equation*}

Since $U_1 \cap U_2$ is disconnected,  the transition maps $g_{U_1U_2}:=\phi_{U_1} \circ \phi^{-1}_{U_2}$ are given by:
\begin{equation*}
\begin{split}
g_{U_1U_2}&: U_1 \cap U_2 \times \mathbb{T}^3 \rightarrow U_1 \cap U_2 \times \mathbb{T}^3,\\
g_{U_1U_2}(m,\widehat{\bm{x}}) 
&=\begin{cases}
(m,\widehat{\bm{x}}), & \text{if $m= e^{2\pi \mathrm{i}t}$, with $t \in (\frac{1}{2}, 1)$}\\[.5mm]
(m,\widehat{(A^t)^{-1}\cdot\bm{x}}), & \text{if $m= e^{2\pi \mathrm{i}t}$, with $t \in (1, \frac                                                                                           {3}{2})$}
\end{cases}
\end{split}
\end{equation*}

As $\theta$ is not exact, we already know that $H^0_{\theta}(\mathcal{S}^0)$ and $H^4_{\theta}(\mathcal{S}^0)$ vanish (see \cite{hr}). Concerning the other Morse-Novikov cohomolgy groups, we prove the following result:
\begin{theorem}\label{inoue} On $\mathcal{S}^0$, for the Lee form $\theta$ given by Tricerri, $H^{1}_{\theta}(\mathcal{S}^0)$ vanishes, $H^{2}_{\theta}(\mathcal{S}^0)\simeq \R$ and $H^{3}_{\theta}(\mathcal{S}^0)\simeq \R$.
\end{theorem}
\begin{proof} The proof is algebraic and the key is to explicitly write the morphism $\beta_*$. 

From now on we denote  by $W_1$ and $W_2$ the two connected components of $U_1 \cap U_2$, namely 
$$W_1=\left\{e^{2\pi \mathrm{i}t} \mid t \in (\tfrac                                                                                           {1}{2}, 1)\right\},\qquad  W_2=\left\{e^{2\pi \mathrm{i}t} \mid t \in  (1, \tfrac                                                                                           {3}{2})\right\}.$$
Consider the functions $ f : U_1 \rightarrow (0, 1)$, $ f(e^{2\pi \mathrm{i}t})=t$ and $\displaystyle g : U_2 \rightarrow (\tfrac                                                                                           {1}{2}, \tfrac{3}{2})$, $\displaystyle g(e^{2\pi \mathrm{i}t})=t$. Then on $U_1$, $\vartheta = df$ and on $U_2$, $\vartheta = dg$. Moreover, we observe that on $W_1$, $f$ and $g$ coincide and on $W_2$, $g=f+1$. Therefore, 
$$\theta = \mathrm{ln} \alpha \cdot dp^*f\,\,\text{on}\,\, U,\,\, \text{and}\,\, \theta = \ln \alpha \cdot dp^*g\,\, \text{on}\,\, V,$$
and hence $\theta$ is exact on these two open sets.

We have the following diagram:
\[
\begin{tikzcd}
H^0_{\theta_{|U}}(U) \oplus H^0_{\theta_{|V}}(V)\arrow{r}{\beta_*} \arrow[swap]{d}{\Phi} & H^0_{\theta_{|U \cap V}}(U \cap V) \arrow{d}{\Psi} \\
\R^2 \arrow{r}{\gamma} & \R^2
\end{tikzcd}
\]
where $\Phi$ and $\Psi$ are the isomorphisms defined as
\begin{equation*}
\begin{split}
\Phi ([\sigma], [\eta])& = (e^{-\mathrm{ln}\alpha p^*f}\sigma, e^{-\mathrm{ln}\alpha p^*g}\eta),\\
\Psi ([\omega])&= (e^{-\mathrm{ln}\alpha p^*f}\omega_{|p^{-1}(W_1)}, e^{-\mathrm{ln}\alpha p^*f}\omega_{|p^{-1}(W_2)}),
\end{split}
\end{equation*}
and $\gamma$ makes the diagram commutative, hence $\gamma(a, b)=(a-b, a- \alpha b)$. 

As $\alpha \neq 1$, $\gamma$ is an isomorphism, and hence  $\beta_*$ is an isomorphism, too. Consequently, the connecting morphism $\delta :  H^0_{\theta_{|U \cap V}}(U \cap V) \rightarrow H^1_{\theta}(\mathcal{S}^0)$ is 0 and we can start the Mayer-Vietoris from $H^1_{\theta}(\mathcal{S}^0)$:
\begin{multline*}
0 \rightarrow H^1_{\theta}(\mathcal{S}^0) \rightarrow H^1_{\theta_{|U}}(U) \oplus H^1_{\theta_{|V}}(V) \rightarrow  H^1_{\theta_{|U \cap V}}(U \cap V)\rightarrow\\
\rightarrow \cdots  \rightarrow H^3_{\theta_{|U \cap V}}(U \cap V) \rightarrow 0
\end{multline*}

We look now at the other morphisms $\beta_*$ linking cohomology groups of degree $i \geq 1$.
\[
\begin{tikzcd}
H^i_{\theta_{|U}}(U) \oplus H^i_{\theta_{|V}}(V)\arrow{r}{\beta_*} \arrow[swap]{d}{\Phi} & H^i_{\theta_{|U \cap V}}(U \cap V) \arrow{d}{\Psi} \\
H^i_{dR}(\mathbb{T}^3) \oplus H^i_{dR}(\mathbb{T}^3) \arrow{r}{\gamma} & H^i_{dR}(\mathbb{T}^3) \oplus H^i_{dR}(\mathbb{T}^3)
\end{tikzcd}
\]
Using the fact that $\theta$ is exact when restricted to $U$ and $V$, the isomorphism $\Phi$ is obtained by the following composition of  isomorphisms:
$$H^i_{\theta_{|U}}(U) \stackrel{f_1}{\longrightarrow}H^{i}_{dR}(U)\stackrel{f_2}{\longrightarrow}
H^{i}_{dR}(U_1 \times \mathbb{T}^3)\stackrel{f_3}{\longrightarrow}H^{i}_{dR}(\mathbb{T}^3),
$$

where $f_1([\sigma])=[e^{-f}\sigma]$, $f_2([\eta])=[(\phi_{U_1})_{*}\eta]$, $f_3([\omega])=[i^*\omega]$ and $i:\mathbb{T}^3 \rightarrow U_1 \times \mathbb{T}^3$ is defined as $i(t) = (m, t)$, for some point $m$ in $U_1$. 

The same holds for $V$, the only difference being that $f_1':H^i_{\theta_{|V}}(V) \rightarrow  H^{i}_{dR}(V)$ is given by $[\sigma] \mapsto [e^{-g}\sigma]$ and $f_2': H^{i}_{dR}(V) \rightarrow H^{i}_{dR}(U_2 \times \mathbb{T}^3)$ is given by $[\eta] \mapsto [(\phi_{U_2})_*\eta]$. 
Thus:
$$\Phi = f_3 \circ f_2 \circ f_1 \oplus  f_3' \circ f_2' \circ f_1'.$$ 

As for $\Psi$, there is a similar sequence, consisting of isomorphisms:
$$
H^i_{\theta_{|U \cap V}}(U \cap V)\stackrel{g_1}{\longrightarrow}H^{i}_{dR}(U \cap V)\stackrel{g_2}{\longrightarrow} H^{i}_{dR}(U \cap V \times \mathbb{T}^3)\stackrel{g_3}{\longrightarrow} H^i_{dR}(\mathbb{T}^3) \oplus H^i_{dR}(\mathbb{T}^3).
$$

Here, the isomorphisms $g_1$, $g_2$ and $g_3$ are given by $[\sigma] \mapsto [e^{-f}\sigma]$, $[\eta] \mapsto [(\phi_U)_*\eta]$ and $[\omega] \mapsto (i_1^*[\omega_{|W_1}], i_2^*[\omega_{|W_2}])$, where $i_1: \mathbb{T}^3 \rightarrow W_1 \times \mathbb{T}^3$ denotes the injection $t \mapsto (m, t)$ for some $m$ in $W_1$ and  $i_2: \mathbb{T}^3 \rightarrow W_2 \times \mathbb{T}^3$, $i_2(t) = (n, t)$ for some point $n$ in $W_2$. We define $\Psi = g_3 \circ g_2 \circ g_1$.  

A straightforward computation shows that $\gamma = \Psi \circ \beta_* \circ \Phi^{-1}$ is given by:
$$([a], [b]) \mapsto ([a -b], [a - \alpha \cdot i_2^*((g_{U_1U_2})_{|W_2})_*\pi^*b]),$$
where $\pi: V \times \mathbb{T}^3 \rightarrow \mathbb{T}^3$ is the projection on the second factor. 

We investigate now the map $i_2^*((g_{U_1U_2})_{|W_2})_*\pi^*: H^{i}_{dR}(\mathbb{T}^3) \rightarrow H^{i}_{dR}(\mathbb{T}^3)$ for $i=1, 2, 3$. It is an easy observation that $$i_2^*((g_{U_1U_2})_{|W_2})_*\pi^*=(\pi \circ (g_{U_1U_2})_{|W_2} \circ i_2)_*.$$

Since $\pi \circ (g_{U_1U_2})_{|W_2} \circ i_2 : \mathbb{T}^3 \rightarrow \mathbb{T}^3$ is given by the matrix $(A^t)^{-1}$, the map induced in homology, $(\pi \circ (g_{U_1U_2})_{|W_2} \circ i_2)_* : H_1(\mathbb{T}^3) \rightarrow H_1(\mathbb{T}^3)$ has the matrix $(A^{t})^{-1}$ in the canonical basis. Therefore, the matrix of the map induced by the pushforward $(\pi \circ (g_{U_1U_2})_{|W_2} \circ i_2)_* : H^1_{dR}(\mathbb{T}^3) \rightarrow H^1_{dR}(\mathbb{T}^3)$ in the canonical basis $\{[dx], [dy], [dz]\}$ is $(((A^{t})^{-1})^{t})^{-1}=A$. 

As a consequence, we obtain the matrix of $\gamma : H^1_{dR}(\mathbb{T}^3) \oplus H^1_{dR}(\mathbb{T}^3) \rightarrow H^1_{dR}(\mathbb{T}^3) \oplus H^1_{dR}(\mathbb{T}^3)$ to be the following:
$$
\left[
\begin{array}{c|c}
I_3 & -I_3 \\
\hline
I_3 & -\alpha \cdot A
\end{array}
\right]
$$

By performing a transformation which keeps the rank constant, namely adding the first three columns to the last three, the matrix above has the same rank as:
$$
\left[
\begin{array}{c|c}
I_3 & O_3 \\
\hline
I_3 & I_3-\alpha \cdot A
\end{array}
\right]
$$
Moreover, this further implies that the rank is controlled by the block $I_{3} - \alpha \cdot A$, which would be a nonsingular matrix if and only if $\frac                                                                                           {1}{\alpha}$ were an eigenvalue of $A$, which is not the case. Hence, $\gamma$ and implicitly $\beta_*$ is an isomorphism, whence from the Mayer-Vietoris sequence, $H^{1}_{\theta}(\mathcal{S}^0)$ has to vanish. 

Since we already know the matrix of $(\pi \circ (g_{U_1U_2})_{|W_2} \circ i_2)_* : H^1_{dR}(\mathbb{T}^3) \rightarrow H^1_{dR}(\mathbb{T}^3)$ is $A$ in the basis $\{[dx], [dy], [dz]\}$, we  can easily compute the matrix of $(\pi \circ (g_{U_1U_2})_{|W_2} \circ i_2)_* : H^2_{dR}(\mathbb{T}^3) \rightarrow H^2_{dR}(\mathbb{T}^3)$ in the basis $\{[dy \wedge dz], [dz \wedge dx], [dx \wedge dz]\}$ to be $(A^*)^t$. Therefore, the matrix of $\gamma : H^2_{dR}(\mathbb{T}^3) \oplus H^2_{dR}(\mathbb{T}^3) \rightarrow H^2_{dR}(\mathbb{T}^3) \oplus H^2_{dR}(\mathbb{T}^3)$ is:

$$
\left[
\begin{array}{c|c}
I_3 & - I_3 \\
\hline
I_3 & -\alpha \cdot (A^*)^{t}
\end{array}
\right]
$$
which by the same arguments as above has the same rank as:
$$
\left[
\begin{array}{c|c}
I_3 & O_3 \\
\hline
I_3 & I_3-\alpha \cdot (A^*)^{t}
\end{array}
\right]
$$
Since $A^{*}=A^{-1} $(because $A$ lives in $\SL_3(\Z)$) and a matrix and its transpose have the same eigenvalues, $(A^*)^t$ has the same eigenvalues as $A^{-1}$, thus $\frac                                                                                           {1}{\alpha}$ is one of them. Therefore, the rank of the block $ I_3-\alpha \cdot (A^*)^{t}$ is 2, because $\frac                                                                                           {1}{\alpha}$ is an eigenvalue of $(A^*)^t$ of multiplicity 1. We infer that the matrix of $\gamma : H^2_{dR}(\mathbb{T}^3) \oplus H^2_{dR}(\mathbb{T}^3) \rightarrow H^2_{dR}(\mathbb{T}^3) \oplus H^2_{dR}(\mathbb{T}^3)$ has rank 5, forcing $\Ker\, \gamma$ to be 1-dimensional and from the Mayer-Vietoris sequence, we obtain $H^2_{\theta}(\mathcal{S}^0) \simeq \R$.

For the final case, when $i=3$, it is straightforward that $(\pi \circ (g_{U_1U_2})_{|W_2} \circ i_2)_* : H^3_{dR}(\mathbb{T}^3) \rightarrow H^3_{dR}(\mathbb{T}^3)$ is given by the multiplication with the determinant of the matrix of $(\pi \circ (g_{U_1U_2})_{|W_2} \circ i_2)_* : H^1_{dR}(\mathbb{T}^3) \rightarrow H^1_{dR}(\mathbb{T}^3)$. In this case, the determinant is 1, hence $\gamma: H^3_{dR}(\mathbb{T}^3) \oplus H^3_{dR}(\mathbb{T}^3) \rightarrow H^3_{dR}(\mathbb{T}^3) \oplus H^3_{dR}(\mathbb{T}^3)$ is given by the $2 \times 2$ -matrix:
$$
\begin{bmatrix}
    1 & -1 \\
    1 & -\alpha\\
   
\end{bmatrix}
$$
and thus it defines an isomorphism. By the Mayer-Vietoris sequence, we obtain: 
$$\dim_{\R}H^3_{\theta}(\mathcal{S}^0)=6-\dim_{\R}\Im (\beta_*: H^2_{\theta_{|U}}(U) \oplus H^2_{\theta_{|V}}(V) \rightarrow H^2_{\theta_{|U \cap V}}(U \cap V))=1$$
In conclusion, $H^3_{\theta}(\mathcal{S}^0) \simeq \R$, $H^2_{\theta}(\mathcal{S}^0) \simeq \R$ and the rest of the Morse-Novikov cohomology groups vanish.
\end{proof}

\begin{remark}\label{bubu}
Since $H^1_{dR}(\mathcal{S}^0) \simeq \R$ (see \cite{in}), $H^1_{dR}(\mathcal{S}^0)=\R [\vartheta]$, hence every closed, but not exact one-form is, up to adding an exact one-form, a multiple of $\vartheta$. Let $\theta_1=t \cdot \vartheta$ with $t \neq 0$. By applying the same method as above and replacing $\mathrm{ln}\, \alpha$ with $t$, we can compute the Morse-Novikov cohomology groups $H^*_{\theta_1}(\mathcal{S}^0)$. Moreover, we observe that for $t \neq \mathrm{ln}\, \alpha$ and $t \neq -\mathrm{ln}\, \alpha$,  the morphisms $\gamma : H^*_{dR}(\mathbb{T}^3) \oplus H^*_{dR}(\mathbb{T}^3) \rightarrow H^*_{dR}(\mathbb{T}^3) \oplus H^*_{dR}(\mathbb{T}^3)$ are in fact isomorphisms, and thus we obtain:
\end{remark}

\begin{corollary}\label{bibi}
For $\theta_1=t \cdot \vartheta$ and $t \neq \mathrm{ln}\, \alpha,  -\mathrm{ln}\, \alpha$, $H^i_{\theta_1}(\mathcal{S}^0)$ vanish for all $i \geq 0$.
\end{corollary}  

We now treat the above two exceptions. 

The case $t =\mathrm{ln}\, \alpha$ is the one discussed above. 

For $t =-\mathrm{ln}\, \alpha$, we apply the following version of Poincar\' e duality:

\begin{proposition}(\cite[Proposition 1.5]{hr}) On a compact oriented $n$-di\-men\-si\-o\-nal manifold $M$, we have the following isomorphism:
$$H^{n-k}_{\eta}(M) \simeq H^k_{-\eta}(M)$$
for any closed one-form $\eta$.
\end{proposition}

Therefore, when $t= -\mathrm{ln}\, \alpha$, we have $\theta_1=-\theta$, $H^1_{\theta_1}(\mathcal{S}^0) \simeq \R$, $H^2_{\theta_1}(\mathcal{S}^0) \simeq \R$ and the rest of the cohomology groups vanish. Thus, we computed the Morse Novikov cohomology of $\mathcal{S}^0$ with respect to any closed one-form.

This result will be useful in Subsection \ref{classif}.

\subsection{Finding generators for $H^2_{\theta}(\mathcal{S}^0)$ and $H^3_{\theta}(\mathcal{S}^0)$}  

Denote by 
$$\Omega :=-\mathrm{i}\left(\frac{dw \wedge d\overline{w}}{w_2^2}+w_2dz \wedge d\bar{z}\right)$$
 the global conformally symplectic  two-form on $\C \times \H$, in the coordinates $(z, w)$, which descends to a two-form $\omega$ on $\mathcal{S}^0$. Notice that $\Omega_1:= -\mathrm{i}\frac{dw \wedge d\overline{w}}{w_2^2}$ and $\Omega_2: = -\mathrm{i}w_2dz \wedge d\bar{z}$ are two-forms which are invariant with respect to the factorization group $G_A$. They descend to $\mathcal{S}^0$ to two forms which we shall denote by $\omega_1$ and $\omega_2$ and we have $\omega=\omega_1+\omega_2$. Tricerri showed that $\omega$ is an LCK form and it is the fundamental two-form of the metric induced by 
$$g=-\mathrm{i}\left(\frac{dw \otimes d\overline{w}}{w_2^2}+w_2dz \otimes d\bar{z}\right)$$ 
on $\mathcal{S}^0$, which we shall denote by $g_1$. Then we have the following:

\begin{proposition}\label{gen} Let $\omega$ be the above defined LCS form of $\mathcal{S}^0$ and $\theta =\frac                                                                                           {dw_2}{w_2}$ its Lee form, as in \ref{inoue}. Then:
\begin{equation*}
\begin{split}
H^{2}_\theta(\mathcal{S}^0)& = \R [\omega]\\
H^{3}_\theta(\mathcal{S}^0)& = \R [\theta \wedge \omega].
\end{split}
\end{equation*}
\end{proposition}

\noindent Before proving these equalities, we define the notion of {\em twisted Laplacian}. Namely, by extending the metric $g_1$ to the space of $k$-forms $\Omega^k(\mathcal{S}^0)$, we consider the Hodge star operator $*: \Omega^k(\mathcal{S}^0) \rightarrow \Omega^{4-k}(\mathcal{S}^0)$, given by $u \wedge * v=g_1(u, v) d\vol$. Note that the real dimension of $\mathcal{S}^0$ is 4.
Then the following operators depending on $\theta$ can be defined (they indeed  make sense on any manifold $M$ endowed with a closed one-form $\theta$, although we shall treat specifically the case of $\mathcal{S}^0$):
\begin{equation*}
\begin{split}
\delta_\theta:&\, \Omega^{k+1}(\mathcal{S}^0) \rightarrow \Omega^{k}(\mathcal{S}^0), \qquad 
\delta_\theta =- *d_{-\theta} *\\
\Delta_\theta:&\, \Omega^k(\mathcal{S}^0) \rightarrow \Omega^{k}(\mathcal{S}^0),\qquad
\Delta_\theta = \delta_\theta d_\theta + d_\theta \delta_\theta
\end{split}
\end{equation*}
\begin{remark} $\delta_\theta$ is the adjoint of $d_\theta$ with respect to the inner product on $ \Omega^k(\mathcal{S}^0)$ given by $\langle \eta, \phi \rangle =\int_{\mathcal{S}^0}\eta \wedge * \phi$.  Observe that $\delta_\theta$ and $\Delta_\theta$ are perturbations of the usual  codifferential and Laplacian operators, which are recovered by replacing $\theta$ with 0. The motivation for introducing the operators twisted with $\theta$ is to develop Hodge theory in the context of working with $d_\theta$ instead of $d$. They were first considered in \cite{va1} for locally conformally  K\"ahler manifolds and later in \cite{gl} in the LCS setting.
\end{remark}

The following analogue of Hodge decomposition holds:
\begin{theorem} {\rm ( \cite{gl})} Let $M$ be a compact manifold, $\theta$ a closed one-form, $\delta_\theta$ and $\Delta_\theta$ defined as above. Then we have an orthogonal decomposition:
\begin{equation}\label{hodge}
\Omega^{k}(M)=\mathcal{H}^k_\theta(M) \oplus d_\theta \Omega^{k-1}(M) \oplus \delta_\theta\Omega^{k+1}(M)
\end{equation}
where $\mathcal{H}^k_\theta(M)=\{\eta \in  \Omega^{k}(M) \mid \Delta_\theta \eta=0\}$.  Moreover,
$$H^{k}_\theta(M) \simeq \mathcal{H}^k_\theta(M).$$
\end{theorem}
Thus, we observe that important properties of the Hodge-de-Rham theory for the operator $d$ are shared by the same theory applied to $d_\theta$. 

We  now give the\\[1mm]
\noindent{\em Proof of \ref{gen}.} 
Since we proved in \ref{inoue} that $H^{2}_\theta$ and $H^{3}_\theta$ are isomorphic to $\R$, it is enough to show that $\omega$ and $\theta \wedge \omega$ are $d_{\theta}$-closed, but not $d_{\theta}$-exact.  

We shall prove that with respect to the Hodge decomposition \eqref{hodge}, the harmonic and the $d_{\theta}$-exact parts of $\omega$ do not vanish. Indeed, a straightforward computation shows that $\Omega_1=d_{\frac                                                                                           {dw_2}{w_2}}\frac                                                                                           {-dw_1}{w_2}$. Since $\frac                                                                                           {-dw_1}{w_2}$ is $G_A$-invariant, it descends to a one-form $\eta$ on $\mathcal{S}^0$ and we have $w_1 = d_{\theta}\eta$. As $\omega$ is the fundamental two-form of the metric $g_1$, which is Hermitian with respect to the complex structure of $\mathcal{S}^0$ induced form the standard one on $\C \times \H$, an easy linear algebra computation (see \cite[p. 31]{gh}) shows that the Riemannian volume form $d\vol$ equals $\frac                                                                                           {\omega^2}{2!}$. In the general case of complex dimension $n$, the volume form $d\vol$ is $\frac                                                                                           {\omega^n}{n!}$. This further implies that $*\omega_2 = \omega_1$. Consequently, 
$$d_{-\theta} * \omega_2 = d_{-\theta}\omega_1 =d{\omega_1}+\theta \wedge \omega_1.$$
 However, $d\Omega_1=0$ and $\frac                                                                                           {dw_2}{w_2} \wedge \Omega_1=0$,  hence $d\omega_1=0$ and $\theta \wedge \omega_1=0$, implying that $\omega_2$ is $\delta_\theta$-closed. Still, one can show that $\Omega_2$ is $d_{\frac                                                                                           {dw_2}{w_2}}$-closed, therefore $\omega_2$ also is $d_{\theta}$-closed. So $\omega_2$ is harmonic with respect to $\Delta_\theta$. Thus, $\omega=\omega_1+\omega_2$ is the Hodge decomposition of $\omega$. We proved in this way that $\omega$ is not $d_\theta$-exact and moreover, $[\omega]=[\omega_2]$ defines a non-vanishing cohomology class in $H^{2}_\theta(M)$. But $H^{2}_\theta(M) \simeq \R$, therefore $H^{2}_\theta(M)= \R[\omega]=\R[\omega_2]$.

As for $H^{3}_\theta(M)$, we first notice that $\theta \wedge \omega$ is $d_\theta$-closed. Indeed, $d_{\theta} (\theta \wedge \omega)=d(d\omega)-\theta \wedge \theta \wedge \omega =0$. In \cite{g}, it was shown that $\Delta_\theta (\theta \wedge \omega) =0$, whence we obtain, as in the case of $\omega$, that $\theta \wedge \omega$ is not $d_\theta$-exact, since its harmonic part is not zero. This implies $H^{3}_\theta(M)=\R[\theta \wedge \omega]$. \blacksquare

\begin{remark} We notice that the alternate sum of the dimensions of the Morse-Novikov cohomology $H^{i}_\theta(\mathcal{S}^0)$ groups is 0, which equals indeed the Euler characteristic of $\mathcal{S}^0$. 
\end{remark}

\begin{remark}\label{bn} The Lee form $\theta$ with respect to which we computed the Morse-Novikov cohomology has important properties: it is nowhere vanishing and it is harmonic (and hence the metric we worked with is a {\em Gauduchon metric}). 

Moreover, the Novikov Betti numbers $b_i^{Nov}$ of $\theta$ vanish, since it has no zeros (see for more details \cite{n1}, \cite{n2}, \cite{f}). It is proven in \cite[Lemma 2]{paj} that if $\eta=a \cdot \tau$, where $\tau$ is an integer closed one-form and $e^a$ is transcendental, then $b_i^{Nov} = \mathrm{dim}_{\R}H^i_{\eta}$.  However, here is not the case, since  $\theta=\mathrm{ln} \alpha \cdot \vartheta$, with $\vartheta$ an integral one-form and $e^{\mathrm{ln}\alpha}$ an algebraic number.

We note again that the Inoue surface $\mathcal{S}^0$ is a mapping torus of the 3-dimensional $\mathbb{T}^3$, which is a {\em contact} manifold. However, the diffeomorphism that defines this mapping torus does not preserve the standard contact structure of $\mathbb{T}^3$. In general, if $(M, \alpha)$ is a contact manifold and $\phi: M \rightarrow M$ is a diffeomorphism preserving $\alpha$, one can consider the mapping torus of $M$ with respect to $\phi$, $M_\phi:= M \times [0, 1]/ (x, 0) \sim (\phi(x), 1)$. Then $M_{\phi}$ admits the LCS form $\omega: = d\alpha - \vartheta \wedge \alpha$, where $\vartheta$ is the integer volume form of the circle. The Morse-Novikov cohomology of $M_{\phi}$ with respect to $\vartheta$ vanishes, as a consequence of \cite[Lemma 2]{paj}. 
\end{remark}

\subsection{Classification of LCK structures on $(\mathcal{S}^0,J)$}\label{classif} Using our previous explicit computation of the Morse Novikov cohomology with respect to all the closed one-forms on $\mathcal{S}^0$ (\ref{bubu}, \ref{bibi}), we are able to describe all the possible Lee forms of an LCK metric on $\mathcal{S}^0$ and classify the LCK metrics. We prove the following result:

\begin{theorem}\label{unique} On the complex surface $(\mathcal{S}^0, J)$, the only possible Lee class for LCK metrics is $[\theta] \in H^1(\mathcal{S}^0)$, where $\theta$ is the Lee form of Tricerri's metric.
\end{theorem} 

For the proof,  we use the structure of solvmanifold of $\mathcal{S}^0$ that we now describe following \cite{s} (see also \cite{k}). 

We consider the following coordinates on $\H \times \C = \{(x+\mathrm{i}t, z) \mid x \in \R, t > 0,  z \in \C \}$. The  group structure on $\H \times \C$ is:
$$(x + \mathrm{i}t, z) \cdot (x' + \mathrm{i}t', z') = (tx' + x + \mathrm{i} t\cdot t', \beta^{\mathrm{log}_{\alpha}t}z'+z)$$
Thus, as a group, $\H \times \C$ can be expressed as a group of matrices as:
$$G = \left\{\begin{bmatrix}
    t       & 0 & 0 & x \\
    0       & \beta^{\mathrm{log}_{\alpha}t} & 0 &  z \\
    0       & 0 & \bar{\beta}^{\mathrm{log}_{\alpha}t} &  \bar{z} \\
     0       & 0 & 0 &  1 
\end{bmatrix} \mid x \in \R, t>0, z \in \C \right\}$$

The group $G$ is solvable. Consider the following lattice:
$$\Gamma= \left\{\begin{bmatrix}
    \alpha^s     & 0 & 0 & x_1 \\
    0       & \beta^s & 0 &  x_2 \\
    0       & 0 & \bar{\beta}^s &  x_3 \\
     0       & 0 & 0 &  1 
\end{bmatrix} \mid s \in \Z \right\}$$

where $\begin{bmatrix}x_1 \\ x_2 \\ x_3 \end{bmatrix}=\begin{bmatrix}a_1 & a_2 & a_3 \\ b_1 & b_2 & b_3 \\ \bar{b}_1 & \bar{b}_2 & \bar{b}_3 \end{bmatrix} \cdot 
\begin{bmatrix} w_1 \\  w_2 \\ w_3 \end{bmatrix}$, and $w_1$, $w_2$, $w_3$ are integers. 

 Then $\mathcal{S}^0$ identifies with $G/\Gamma$.

The solvable Lie algebra corresponding to $G$ is:
\begin{equation*}
\begin{split}
\mathfrak{g}=\span \big\{A, X, Y_1, Y_2 \mid [A, X]&=-2rX,\\
 [A, Y_1]&=rY_1+s Y_2, [A, Y_2]=rY_2-s Y_1\big\},
 \end{split}
 \end{equation*}
where $r=-\frac{\ln \alpha}{2}$ and $\beta=e^{r+\mathrm{i}s}$.

One can define a left invariant complex structure $J_0$ on $G/\Gamma$  by $JA = X$, $JY_1=Y_2$. 

The manifolds $(\mathcal{S}^0, J)$ and $(G/\Gamma, J_0)$ are biholomorphic. {\em Via}  this biholomorphism, the dual base of $\{A, X, Y_1, Y_2\}$ consisting of left invariant one-forms is $\{\vartheta, x, y_1, y_2\}$, where $\vartheta$ was defined in Subsection \ref{descriere}, and $x$, $y_1$ and $y_2$ satisfy:
\begin{equation}\label{ec}
dx=2r\vartheta \wedge x, dy_1=-r\vartheta \wedge y_1 + s \vartheta \wedge y_2, dy_2 = -r \vartheta \wedge y_2 - s\vartheta\wedge y_1
\end{equation}

\noindent The LCK form given by Tricerri can also be written as:
$$\omega = - \vartheta \wedge x - y_1 \wedge y_2$$
with the Lee form $\theta=-2r \vartheta=\ln \alpha \cdot \vartheta$.

\noindent{\em Proof of \ref{unique}:}
We prove that $\theta$ and its cohomologous one-forms are the only possible Lee forms for LCK metrics on the Inoue surface $\mathcal{S}^0$. 
Indeed, let $\theta_1$ be another possible Lee form for an LCK metric. As $b_1(\mathcal{S}^0)=1$, $\theta_1$ is, up to a an exact one-form, in one of the following three cases:
\begin{enumerate} 
\item[{\bf 1.}]$\theta_1=t \vartheta$, with $t$ different from $\ln \alpha$ or $- \ln \alpha$.
\item[{\bf 2.}] $\theta_1=\ln \alpha \cdot \vartheta$ and thus coincides with the Lee form of Tricerri's  LCK metric.
\item[{\bf 3.}] $\theta_1=-\ln \alpha \cdot \vartheta$.
\end{enumerate}
Before separately discusing these three cases, we need two general results:

\begin{claim}  In all of the  three cases, $\theta_1$ is left invariant.
\end{claim}

Indeed, this follows from the left invariance of  $\vartheta$.

\begin{claim}\label{claimm}
If $\omega_1$ is an LCK form with the Lee form $\theta_1$ (no matter in which of the three possibilities above), then it cannot be $d_{\theta_1}$-exact.
\end{claim}

\begin{proof} Indeed, by contradiction, assume that $\omega_1=d_{\theta_1}\eta = d\eta - \theta_1 \wedge \eta$.
It was proven in \cite [Proposition 1.2]{s} that one can further find a left invariant form $\eta_0$ such that $\omega_0: =d_{\theta_1}\eta_0$ is still an LCK form. 
We obtain the following:
$$\omega_0(Y_1, Y_2)=d\eta_0 (Y_1, Y_2) -\theta_1 \wedge \eta_0 (Y_1, Y_2).$$
However, since $\theta_1=t \vartheta$ and $\vartheta(Y_1)=\vartheta(Y_2)=0$, we have $$\omega_0(Y_1, Y_2)=d\eta_0(Y_1, Y_2)=Y_1(\eta_0(Y_2)) - Y_2(\eta_0(Y_1))-\eta_0([Y_1, Y_2]).$$
 But since $\eta_0$, $Y_1$ and $Y_2$ are left invariant, the first two terms in the right hand side vanish and thus,
$$\omega_0(Y_1, Y_2)=-\eta_0([Y_1, Y_2]).$$
 By relations \eqref{ec}, we obtain $[Y_1, Y_2]=0$, and hence $\omega_0(Y_1, Y_2)=0$. This contradicts the fact that $\omega_0$ is the K\"ahler form of a Hermitian metric $g_0$, since $\omega_0(Y_1, Y_2)=g_0(J Y_1, Y_2)=g_0(Y_2, Y_2) \neq 0$.

Therefore, an LCK form on $\mathcal{S}^0$ with Lee form $\theta_1$ is not allowed to be $d_{\theta_1}$-exact. 
\end{proof}

Now, for $\theta_1$ as in {\bf Case 1} above,  by \ref{bibi}  the Morse-Novikov cohomology vanishes, so the LCK metric would be exact, which we saw it is impossible. So, {\bf Case 1} is excluded.

For {\bf Case 2}, $\theta_1=\theta$ and we showed that $H^2_{\theta}(\mathcal{S}^0)=\R[\omega]$. Therefore, we see that any other LCK metric $\omega_1$ has to be of the type $r\omega + d_{\theta}\eta$. Note that $r \neq 0$, since by \ref{claimm}, an LCK form on $\mathcal{S}^0$ cannot be $d_{\theta}$-exact.

{\bf Case 3} is a bit more involved. Now $\theta_1=-\theta$ and we already shown in Subsection \ref{calcul} that $H^1_{\theta_1}(\mathcal{S}^0) \simeq \R$, $H^2_{\theta_1}(\mathcal{S}^0) \simeq \R$ and the other groups vanish. However, there is still no LCK metric with $\theta_1$ as Lee form and in order to argue that, we use the results in \cite{ad}.

Let $\eta$ be a closed one-form on a complex compact surface $(M, J)$ and let $\mathcal{L_\eta}: =L_\eta \otimes \C$, where $L_\eta$ is the line bundle associated to $\eta$, as presented in the first section. Let $g$ be a Gauduchon metric with the corresponding Lee form denoted by $\theta^g$.

\begin{definition} The degree of $\mathcal{L}_\eta$ with respect to the Gauduchon metric $g$ is defined to be:
$$\mathrm{deg}_g\mathcal{L}_\eta=-\frac{1}{2 \pi}\int_Mg(\theta^g, \eta)v_g$$
where $v_g$ is the volume form of $g$.
\end{definition}

It was shown in \cite[Lemma 4.1]{ad} that on a compact complex surface $(M, J)$ with $b_1(M)=1$, the sign of $\mathrm{deg}_g\mathcal{L}_\eta$ does not depend on the choice of the Gauduchon metric on $M$. In the case of $(\mathcal{S}^0, J)$, we can choose as Gauduchon metric Tricerri's metric $\omega$, with its corresponding Lee form $\theta$ (see also \ref{bn}). Therefore, the sign of $\mathcal{L}_{\theta_1}$ is the sign of $-\frac                                                                                           {1}{2 \pi}\int_Mg(\theta, \theta_1)v_g$ and it is positive, since $\theta_1=-\theta$.
Nevertheless, in \cite[Proposition 4. 3]{ad}, it is proved that on a compact complex surface $(M, J)$ with $b_1(M)=1$, if $\eta$ is the Lee form of an LCK metric or more general, the Lee form of an LCS form which tames $J$, then the degree of $\mathcal{L}_\eta$ is negative, thus excluding $\theta_1$ from the possible Lee forms. 

Hence $\mathcal{C}(\mathcal{S}^0)=\{[\theta]\}$ and the proof is complete.
\endproof

\begin{remark} With the same proof, we obtain $\mathcal{T}(\mathcal{S}^0)=\{[\theta]\}$, since we only use the $d_{\theta_1}$-closedness of $\omega_1$ and the positiveness $\omega_1(X, JX) >0$ for any non-zero $X$ .
\end{remark}

\begin{remark} There are LCK metrics in $\mathcal{S}^0$ which are not left invariant. For instance, consider the form 
$$\Omega_1=-\mathrm{i}(e^{\mathrm{sin}(2 \pi \mathrm{log}_{\alpha}w_2)}\frac{dw \wedge d\bar{w}}{w_2^2} + w_2dz \wedge d\bar{z}).$$
 It is straightforward that $\Omega_1$ is LCK with the Lee form $\theta=\frac{dw_2}{w_2}$ and it is not left invariant. Note that also $\omega$ and $\Omega_1$ are two LCK metrics with the same Lee form, but which are not conformal.
\end{remark}

\section{Morse-Novikov cohomology of the Inoue surfaces $\mathcal{S}^+$ and $\mathcal{S}^-$}

We describe the Inoue surfaces $\mathcal{S}^+$ and $\mathcal{S}^-$ and compute their Morse-Novikov cohomology groups.
\subsection{The complex surface $\mathcal{S}^+$}
Let $N=(n_{ij}) \in \mathrm{SL}_2(\Z)$ be a  matrix with real eigenvalues $\alpha >1$ and $1/\alpha$ and $(a_1, a_2)^{t}, (b_1, b_2)^{t}$  real eigenvectors corresponding to $\alpha$ and $1/\alpha$. Let us fix some integers $p, q, r$ with $r \neq 0$ and a complex number $z$. Let $e_1, e_2$ be defined as
$$e_i = \tfrac{1}{2}n_{i1}(n_{i1}-1)a_1b_1 + \tfrac{1}{2}n_{i2}(n_{i2}-1)a_2b_2 + n_{i1}n_{i2}b_1a_2$$ 
and $c_1, c_2$ defined by
$$(c_1, c_2)=(c_1, c_2) \cdot N^t + (e_1, e_2) + \frac{b_1a_2-b_2a_1}{r}(p, q).$$
We denote by $G^+_{N, p, q, r, z}$ the group of affine transformations of $\R^3 \times \R^+$ (we consider here the coordinates $x, y, w, t$, where $t>0$) generated by the following:
\begin{equation*}
\begin{split}
g_0(x, y, w, t)& = (x + \Re z, y + \Im z, \alpha w, \alpha t)\\
g_i(x, y, w, t) & =(x+b_iw+c_i, y+b_it, w+a_i, t), \qquad i=1, 2\\
g_3(x, y, w, t) & =(x+ \frac{b_1a_2-b_2a_1}{r}, y, w, t).
\end{split}
\end{equation*}

We define $\mathcal{S}^+_{N, p, q, r, z}$ to be $\R^3 \times \R^+/G^+_{N, p, q, r, z}$ and denote by $[x, y, w, t]$ the class of $(x, y, w, t)$.

The transformations $g_0, g_1, g_2$ and $g_3$ satisfy the relations:
\begin{equation*}
\begin{split}
g_3g_i& = g_ig_3,  \qquad {\rm {for}} \qquad i=0, 1, 2,\\
g_1^{-1}g_2^{-1}g_1g_2& =g_3^r\\
g_0g_1g_0^{-1}& =g_1^{n_{11}}g_2^{n_{12}}g_3^{p},\\
 g_0g_2g_0^{-1}& =g_1^{n_{21}}g_2^{n_{22}}g_3^{q}.
\end{split}
\end{equation*}

\begin{remark} Tricerri showed in \cite{t} that for $z \in \R$, the complex surface $\mathcal{S}^+_{N, p, q, r, z}$ carries an LCK metric given by the  two-form 
$$\omega = 2 \frac{1+y^2}{t^2}dt \wedge dw - 2\frac{y}{t}(dt \wedge dx + dy \wedge dw) +2dy \wedge dx.$$
 In this case, the Lee form is $\theta=\frac{dt}{t}$. Belgun proved in \cite{bel} that for $z \in \C \setminus \R$, $S^+_{N, p, q, r, z}$ does not carry an LCK metric. We shall work with the parameter $z=0$, since $\mathcal{S}^+_{N, p, q, r, z}$ analytically deforms to $\mathcal{S}^+_{N, p, q, r, 0}$. Moreover, we shall use the more convenient notations $\mathcal{S}^+$ and $G^+$ from now on. 
\end{remark}

We consider $p: \R^3 \times \R^+ \rightarrow S^1$, $p(x, y, w, t)= e^{2\pi \mathrm{i}\mathrm{log}_{\alpha}t}$. This map descends to a submersion $\pi : \mathcal{S}^+ \rightarrow S^1$, which endowes $\mathcal{S}^+$ with the structure of a fiber bundle with fiber $F$. Therefore, $\mathcal{S}^+$ is mapping torus of $F$.  The fiber over the point $s=e^{2\pi \mathrm{i}\mathrm{log}_{\alpha}t} \in S^1$ is described by
$$F \simeq \pi^{-1} (s) = \{[x, y, w, t] \mid x, y, w \in \R \}$$ 
Following \cite{in}, we denote by $\Gamma_t$ the normal subgroup of $G^+$ generated by $g_1$, $g_2$ and $g_3$ with $t$ fixed. Then $\pi^{-1}(s) \simeq \R^3/\Gamma_t$. We trivialize $\mathcal{S}^+$ with the open sets:
\begin{equation*}
\begin{split}
U&=\{[x, y, w, t] \mid x, y, w \in \R, t \in (1, \alpha)\}\\
V&=\{[x, y, w, t] \mid x, y, w \in \R, t  \in (\sqrt{\alpha}, \alpha \sqrt{\alpha})\}.
\end{split}
\end{equation*}

The generic fiber is $F = \{[x, y, w, {\sqrt[4]{\alpha}}] \mid x, y, w \in \R\} \simeq \R^3 /\Gamma_{\sqrt[4]{\alpha}}$. The trivialization of $\mathcal{S}^+$ over $U$ and $V$ is given by the diffeomorphisms:
$$\phi_U: U \rightarrow U_1 \times F$$
$$\phi_{U}([x, y, w, t])= (e^{2\pi \mathrm{i} \mathrm{log}_{\alpha}t}, [x, \tfrac{\sqrt[4]{\alpha}}{t}y, w, \sqrt[4]{\alpha}])$$
$$\phi_V: V \rightarrow U_2 \times F$$
$$\phi_{V}([x, y, w, t])= (e^{2\pi \mathrm{i} \mathrm{log}_{\alpha}t}, [x, \tfrac{\sqrt[4]{\alpha}}{t}y, w, \sqrt[4]{\alpha}])$$

The transition function $g_{UV}= \phi_U \circ \phi_V^{-1}: U \cap V \times F \rightarrow U \cap V \times F$ is given by:
$$g_{UV}(m, [x, y, w, \sqrt[4]{\alpha}]) = \begin{cases} (m, [x,  y, w, \sqrt[4]{\alpha}]), & \text{for}\,  m \in W_1 \\
(m, [x, \alpha y, \tfrac{1}{\alpha} w, \sqrt[4]{\alpha}]), & \text{for}\, m \in W_2
\end{cases}.
$$

Differentiably, $F$ is circle bundle over the two-torus $\mathbb{T}^2 = \R^2 / \Z^2$. Indeed, let $p: F \rightarrow \mathbb{T}^2$, defined by

$$p (\widehat{(x, y, w)}) =\stackon[-8pt]{\left(\frac{a_2y - b_2 \sqrt[4]{\alpha}w}{\sqrt[4]{\alpha}(b_1a_2 - a_1b_2)}, \frac{-a_1 y + b_1 \sqrt[4]{\alpha} w}{\sqrt[4]{\alpha}(b_1a_2 - a_1b_2)}\right)}{\vstretch{2.2}{\hstretch{6}{\widehat{\phantom{\;\;\;\;\;\;\;\;}}}}}$$
Then $p$ is a well defined submersion onto $\mathbb{T}^2$, whose fiber is the circle $\R / (x \sim x+ \frac{b_1a_2-a_1b_2}{r})$.

We shall compute the Morse-Novikov cohomology groups of $\mathcal{S}^+$ with respect to the Lee form of Tricceri's metric:
$$\theta=\frac{dt}{t}=\mathrm{ln} \alpha \cdot \pi^* \vartheta.$$
 Since $b_1(\mathcal{S}^+)=1$, every closed one-form is, up to an exact factor, a multiple of $\theta$. 

By applying the Mayer-Vietoris sequence to the open sets $U$ and $V$, the following is an exact sequence:
\begin{multline*}
0 \rightarrow H^0_{\theta}(\mathcal{S}^+) \rightarrow H^0_{\theta_{|U}}(U) \oplus H^0_{\theta_{|V}}(V) \rightarrow  H^0_{\theta_{|U \cap V}}(U \cap V)\rightarrow\\
\rightarrow \cdots  \rightarrow H^3_{\theta_{|U \cap V}}(U \cap V) \rightarrow 0
\end{multline*}

As in the case of $\mathcal{S}^0$, we shall be interested in the morphisms 
$H^i_{\theta_{|U}}(U) \oplus H^i_{\theta_{|V}}(V)\xrightarrow{\beta_*} H^i_{\theta_{|U \cap V}}(U \cap V)$, where $i=0, 1, 2, 3$, which further yield the morphisms $\gamma_i: H^i_{dR}(F) \oplus  H^i_{dR}(F) \rightarrow H^i_{dR}(F) \oplus  H^i_{dR}(F)$. The case $i=0$ is identical to the one in $S^0$, yielding an isomorphism which allows us to consider the Mayer-Vietoris sequence from $H^1_{\theta}(\mathcal{S}^+$). We want to write down explicitely the morphisms $((g_{UV})_{|W_2})_*: H^i_{dR}(F) \rightarrow H^i_{dR}(F)$.

According to \cite{fgg}, for any circle bundle $F \xrightarrow{p} \mathbb{T}^2$, except for the trivial one, $H^1_{dR}(F) \simeq \R^2$,  $H^2_{dR}(F) \simeq \R^2$. Moreover, the generators are given by:
$$H^1_{dR}(F) =\R \langle p^*[\eta_1], p^* [\eta_2] \rangle$$
$$H^2_{dR}(F) =\R \langle [\eta \wedge p^*\eta_1], [\eta \wedge p^*\eta_2] \rangle$$
where $\eta_1$ and $\eta_2$ are integral closed one-forms on $\mathbb{T}^2$, such that $[\eta_1 \wedge \eta_2]$ generates $H^2(\mathbb{T}^2)$ and $\eta$ is the curvature form of $F$, satisfying $d\eta=p^*\eta_1 \wedge p^*\eta_2$. 

Let us take $\eta_1=dy$ and $\eta_2=dw$ written in the coordinates on $\R^2$. Since $(g_{UV})_{|W_2}:\R^3/\Gamma_{\sqrt[4]{\alpha}} \rightarrow \R^3/\Gamma_{\sqrt[4]{\alpha}}$ is given by 
$$(g_{UV})_{|W_2}(\widehat{(x, y, z)}) = \stackon[0.02pt]{(x, \alpha \cdot y, \tfrac{1}{\alpha} \cdot w)}{\vstretch{1.2}{\hstretch{6}{\widehat{\phantom{\;\;\;\;\;}}}}},$$
 we obtain the following: 

\begin{claim}
The matrices of $((g_{UV})_{|W_2})_* : H^{i}_{dR}(F) \rightarrow H^i_{dR}(F)$, for $i=1, 2$ in the basis $\{p^* [\eta_1], p^* [\eta_2]\}$ and $\{[\eta \wedge p^*\eta_1], [\eta \wedge p^*\eta_2]\}$ are both equal to:
$$ 
A=T^{-1} \begin{bmatrix}
   \frac{1}{\alpha} & 0\\
    0 & \alpha
   \end{bmatrix} T
$$
where $T =  \begin{bmatrix}
   b_1 \sqrt[4]{\alpha}  &  b_2 \sqrt[4]{\alpha}\\
    a_1 & a_2
   \end{bmatrix}$. 
\end{claim}
 
Similar to the computation for $\mathcal{S}^0$, we can prove that the morphism 
$$\beta_*: H^i_{\theta_{|U}}(U) \oplus H^i_{\theta_{|V}}(V) \rightarrow  H^i_{\theta_{|U \cap V}}(U \cap V)$$
 yields a linear application 
 $$\gamma_i: H^i_{dR}(F) \oplus H^i_{dR}(F) \rightarrow  H^i_{dR}(F) \oplus H^i_{dR}(F),$$
  whose matrix in the case $i=1, 2$ is $$\left[
\begin{array}{c|c}
I_2 & -I_2 \\
\hline
I_2 & -\alpha \cdot A
\end{array}
\right],$$  
whilst for $i=3$ it is 
$$\begin{bmatrix}
    1 & -1\\
    1 & -\alpha\\
   \end{bmatrix}.$$
Since $\frac{1}{\alpha}$ is an eigenvalue of $A^{-1}$, we conclude that 
\begin{claim}
 $\Ker\, \gamma_1$ and $\Ker\, \gamma_2$ are both one-dimensional and $\gamma_3$ is an isomorphism. 
\end{claim}

Now, applying the Mayer-Vietoris sequence, we obtain:

\begin{proposition} For $\mathcal{S}^+= \mathcal{S}^+_{N, p, q, r, 0}$ and Tricerri's Lee form  $\theta=\frac{dt}{t}$, the following holds:
 $$H^1_{\theta}(\mathcal{S}^+) \simeq \R, \qquad H^2_{\theta}(\mathcal{S}^+) \simeq \R^2, \qquad H^3_{\theta}(\mathcal{S}^+) \simeq \R,$$
 and both $H^0_{\theta}(\mathcal{S}^+)$ and $H^4_{\theta}(\mathcal{S}^+)$ vanish.
\end{proposition}  

\begin{remark}\label{parametru}
Note that the result above holds also for $\mathcal{S}^+_{N, p, q, r, z}$ with parameter $z \in \C$, but with respect to the one-form which corresponds to $\frac{dt}{t}$ {\em via} the diffeomorphism between $\mathcal{S}^+_{N, p, q, r, z}$ and $\mathcal{S}^+_{N, p, q, r, 0}$.
\end{remark}

\begin{remark} For $\theta_1=-\theta$, we obtain by Poincar\'e duality the cohomology groups 
$$H^1_{\theta_1}(\mathcal{S}^+) \simeq \R,\quad H^2_{\theta_1}(\mathcal{S}^+) \simeq \R^2,\quad H^3_{\theta_1}(\mathcal{S}^+) \simeq \R.$$
 Since $b_1(\mathcal{S}^+)=1$, by the same argument used in \ref{bibi}, $H^i_{\theta_1}(\mathcal{S}^+)=0$, if $\theta_1 \neq \pm \theta$.
\end{remark}

\subsection{Finding generators for $H^i_{\theta}(\mathcal{S}^+_{N, p, q, r, 0})$}

As in the case of $\mathcal{S}^0$, we shall apply the twisted Hodge decomposition.

Let $\zeta = \frac{ydt}{t}-dy$. Then $\zeta$ is an invariant form on $\C\times \H$ with respect to the action of the group $G^+$ and defines a one-form on $\mathcal{S^+}$.  We prove that $\zeta$ is $\Delta_\theta$-harmonic. Indeed, since with respect to  Tricerri's metric, we have the orthonormal invariant basis of one-forms:
$$\{\frac{ydt}{t}-dy, \frac{ydw}{t}-dx, \frac{dt}{t}, \frac{dw}{dt}\}$$
and
$$d\mathrm{vol}=\frac{\omega \wedge \omega}{2}= \frac{dy \wedge dx \wedge dt \wedge dw}{w_2^2},$$ 
we obtain:
$$\delta_{\theta} \zeta = - * d_{-\theta}*\zeta=-*d_{-\theta}\frac{dx \wedge dw \wedge dt}{t^2}=-*0=0$$
Moreover, it is easy to see that also $d_{\theta}\zeta=0$ holds, therefore $\zeta$ is harmonic with respect to  $\Delta_{\theta}$, hence its class $[\zeta] \in H^{1}_{\theta}(\mathcal{S}^+)$ doesn't vanish and thus
$$H^{1}_{\theta}(\mathcal{S}^+)=\R [\zeta].$$

We observe that $\omega=2\frac{dt \wedge dw}{t^2} + h$, where 
$$h: = 2\frac{y^2}{t^2}dt \wedge dw - 2\frac{y}{t}(dt \wedge dx + dy \wedge dw) +2dy \wedge dx.$$
  Then $h$ is a well defined form on $\mathcal{S}^+$, which is $d_{\theta}$-closed. Moreover, 
$$\delta_{\theta} h=-*d_{-\theta}*h=-2*d_{-\theta}\frac{dt \wedge dw}{t^2}=-2*0=0.$$
Therefore, $h$ is harmonic and since 
$$\frac{dt \wedge dw}{t^2}=d_{\theta}(-\frac{dw}{t}),$$
 $\omega = d_{\theta}(-\frac{dw}{t})+h$ is the twisted Hodge decomposition of $\omega$, it implies that $H^2(\mathcal{S}^+)\ni [\omega]=[h] \neq 0$.

Let now $\tau:= \frac{dy \wedge dt}{t}$. This is a well defined form on $\mathcal{S}^+$, which is $d_{\theta}$-closed.
$$\delta_{\theta} \tau = -*d_{-\theta}* \tau = -* d_{-\theta}\frac{dx \wedge dw}{t}=-*0=0.$$
So $\tau$ is also harmonic and since it is not a multiple of $h$, we get 
 $$H^2_{\theta}(\mathcal{S}^+)= \R \langle [\tau], [h] \rangle.$$
Lastly, we know from \cite{g}, that $\theta \wedge \omega$ is harmonic, hence $0 \neq [\theta \wedge \omega] \in H^{3}_{\theta}(\mathcal{S}^+)$ and
$$H^{3}_{\theta}(\mathcal{S}^+) = \R [\theta \wedge \omega].$$

Note that \ref{parametru} applies here too.

\subsection{The complex surface $\mathcal{S}^-$}

Let $N \in \mathrm{GL}_2(\Z)$ with $\mathrm{det}N=-1$ and eigenvalues $\alpha > 1$ and $-\frac{1}{\alpha}$. We consider $(a_1, a_2)^t$ and $(b_1, b_2)^t$  real eigenvectors corresponding to $\alpha$ and $-\frac{1}{\alpha}$ and let $c_1, c_2$ be defined by
$$-(c_1, c_2)=(c_1, c_2) \cdot N^t + (e_1, e_2) + \frac{b_1a_2-b_2a_1}{r}(p, q),$$
where $e_1$ and $e_2$ are defined as in the case of $\mathcal{S}^+$ and $p$, $q$, $r$ ($r \neq 0$) are integers. We denote by $G^-_{N, p, q, r}$ the group of affine transformations of $\R^3 \times \R^+$ generated by:
\begin{equation*}
\begin{split}
g_0(x, y, w, t)& = (-x, -y, \alpha w, \alpha t)\\
g_i(x, y, w, t) & =(x+b_iw+c_i, y+b_it, w+a_i, t), \qquad i=1, 2\\
g_3(x, y, w, t) & =(x+ \frac{b_1a_2-b_2a_1}{r}, y, w, t).
\end{split}
\end{equation*}
The complex surface $\mathcal{S}^-_{N, p, q, r}$ is defined to be $\R^3 \times \R^+/G^-_{N, p, q, r}$. Since $\langle g_0^2, g_1, g_2, g_3\rangle = G^+_{N^2, p_1, p_2, r, 0}$ for some integer numbers $p_1$ and $p_2$, the following is immediate:

\begin{claim} (\cite{in})
 $\mathcal{S}^+_{N^2, p_1, q_2, r, 0}$ is a double unramified covering of $\mathcal{S}^-_{N, p, q, r}$.
\end{claim}

 Let 
$$\omega_1 = 2 \frac{1+y^2}{t^2}dt \wedge dw - 2\frac{y}{t}(dt \wedge dx + dy \wedge dw) +2dy \wedge dx.$$ 
Then $\omega_1$ defines an LCK metric on $\mathcal{S}^-_{N, p, q, r, 0}$ with  Lee form $\theta_1=\frac{dt}{t}$. We are interested in the Morse-Novikov cohomology of $\mathcal{S}^-_{N, p, q, r, 0}$ with respect to $\theta_1$.

Let $\pi : \mathcal{S}^+_{N^2, p_1, q_2, r, 0} \rightarrow \mathcal{S}^-_{N, p, q, r}$ be the covering map given by $[x, y, w, t] \mapsto [[x, y, w, t]]$, where we make distictions between the equivalence classes with respect to  the two factorizations. In other words, $$\mathcal{S}^-_{N, p, q, r} \simeq \mathcal{S}^+_{N^2, p_1, q_2, r, 0}/\langle \mathrm{id}, \sigma \rangle,$$
 where $\sigma : \mathcal{S}^+_{N^2, p_1, q_2, r, 0} \rightarrow \mathcal{S}^+_{N^2, p_1, q_2, r, 0}$ is the involution 
$$\sigma ([x, y, w, t] = [-x, -y, \alpha w, \alpha t]).$$

Since $\pi^*\theta_1=\theta$,  like in the case of de Rham cohomology, there is an injection $H^i_{\theta_1}(\mathcal{S}^-_{N, p, q, r}) \hookrightarrow H^i_{\theta}(\mathcal{S}^+_{N^2, p, q, r})$.  Therefore, the following inequalities hold:
$$\mathrm{dim}_{\R}H^1_{\theta_1}(\mathcal{S}^-_{N, p, q, r}) \leq 1, \qquad \mathrm{dim}_{\R}H^2_{\theta_1}(\mathcal{S}^-_{N, p, q, r}) \leq 2, \qquad \mathrm{dim}_{\R}H^3_{\theta_1}(\mathcal{S}^-_{N, p, q, r}) \leq 1.$$

As the other Inoue surfaces, $\mathcal{S}^-_{N, p, q, r}$ is also a fiber bundle over $S^1$. Let $\pi: \mathcal{S}^-_{N, p, q, r} \rightarrow S^1$ be the submersion given by $[x, y, w, t] \mapsto e^{2\pi \mathrm{log}_{\alpha}t}$, which endows $ \mathcal{S}^-_{N, p, q, r}$ with the structure of a fiber bundle over $S^1$ whose fiber is $F=\R^3/\Gamma_{\sqrt[4]{\alpha}}$ and transition function $g_{UV} : U \cap V\times F  \rightarrow U \cap V \times F$ is given by:
$$g_{UV}(m , [x, y, w, \sqrt[4]{\alpha}]) = \begin{cases} (m, [x,  y, w, \sqrt[4]{\alpha}]), &   m \in W_1 \\
(m, [- x, - \alpha y, \tfrac{1}{\alpha} w, \sqrt[4]{\alpha}]), &  m \in W_2
\end{cases}.
$$

We need to check as in the other two examples of Inoue surfaces the matrices of $((g_{UV})_{|W_2})_*: H^i_{dR}(F) \rightarrow  H^i_{dR}(F)$. 
For $i=1$, in the base $\{p^*[\eta_1], p^*[\eta_2]\}$, $((g_{UV})_{|W_2})_*$ corresponds to the matrix
$$A=T^{-1}\begin{bmatrix}
   -\frac{1}{\alpha} & 0\\
    0 & \alpha\\
   \end{bmatrix}T$$
where $T$ is the same as in the case of $\mathcal{S}^+$ and for $i=2$, in the base $\{[\eta \wedge p^*\eta_1], [\eta \wedge p^*\eta_2]\}$, $((g_{UV})_{|W_2})_*$ corresponds to the matrix:
$$B=T^{-1}\begin{bmatrix}
   \frac{1}{\alpha} & 0\\
    0 & - \alpha\\
   \end{bmatrix}T.$$

They further yield in the Mayer-Vietoris sequence the linear applications $\gamma_i: H^i_{dR}(F) \oplus H^i_{dR}(F) \rightarrow H^i_{dR}(F) \oplus H^i_{dR}(F)$, whose matrices are for $i=1$:

$$\left[
\begin{array}{c|c}
I_2 & -I_2 \\
\hline
I_2 & -\alpha \cdot A
\end{array}
\right],$$ 
for $i=2$:
$$\left[
\begin{array}{c|c}
I_2 & -I_2 \\
\hline
I_2 & -\alpha \cdot B
\end{array}
\right]$$
and for $i=3$:
$$\begin{bmatrix}
   1 & -1\\
    1 & - \alpha
\end{bmatrix}.$$

As $\alpha$ is an eigenvalue of $B^{-1}$, but not of $A^{-1}$, $\gamma_1$ and $\gamma_3$ are isomorphisms and $\Ker \, \gamma_2$ is one-dimensional. From the Mayer-Vietoris sequences, we obtain:

\begin{proposition} For every $\mathcal{S}^{-}=\mathcal{S}^-_{N, p, q, r}$ and $\theta=\frac{dt}{t}$, the following holds:
 $$H^2_{\theta}(\mathcal{S}^-) \simeq \R, \qquad H^3_{\theta}(\mathcal{S}^-) \simeq \R,$$
 and $H^i_{\theta}(\mathcal{S}^-)$ vanish, for all $i=0, 1, 4$.
\end{proposition}

Since the same proof works as for $\mathcal{S}^+$ to show that $\omega_1$ is not $d_{\theta}$-exact and we know from \cite{g} that $\theta \wedge \omega$ is harmonic, we get:
\begin{corollary} $H^2_{\theta}(\mathcal{S}^-) = \R [\omega_1]$ and $H^3_{\theta}(\mathcal{S}^-) = \R [\theta \wedge \omega_1]$.
\end{corollary}

\begin{remark} Note that the Inoue surfaces $\mathcal{S}^+$ and $\mathcal{S}^-$ have the same Betti numbers. However, $H^*_\theta(\mathcal{S}^+)$ and $H^*_\theta(\mathcal{S}^-)$ differ, meaning that Morse-Novikov cohomology is a better tool than de Rham cohomology to distinguish between $\mathcal{S}^+$ and $\mathcal{S}^-$. 
\end{remark}

\begin{remark} In \cite{m}, Morse-Novikov cohomology of solvmanifolds w.r.t left-invariant forms is considered. Moreover, it is proven that if the solvmanifold is completely solvable, then the Morse-Novikov cohomology coincides with the invariant one and the author shows that by multiplying an invariant Lee form with any real number, one obtains vanishing of the cohomology, except for some finite number of values.  However, $\mathcal{S}^0$ is not completely solvable, but $\mathcal{S}^+$ is, therefore \cite[Corollary 2.3]{m} can be applied for computing the Morse-Novikov cohomology of $\mathcal{S}^+$ and the argument closely  resembles our computations. The surface $\mathcal{S}^-$ is only up to a double unbranched cover a completely solvable solvmanifold (see \cite{has}, the double cover is $\mathcal{S}^+$), therefore, in order to apply \cite[Corollary 2.3]{m} for $\mathcal{S}^-$, one has to apply it for $\mathcal{S}^+$ and then take the $\Z_2$-invariant cohomology, which eventually gives the Morse-Novikov cohomology of $\mathcal{S}^-$.
\end{remark}

\subsubsection{Nonexistence of LCK metrics with potential}

\hfill

\begin{remark}  According to \cite{ad}, the set of Lee classes of LCK metrics on $\mathcal{S}^+$ and $\mathcal{S}^-$ has at most one element, namely for $\mathcal{S} = \mathcal{S}^+_{N, p, q, r, z}$, $\mathcal{S}^-_{N, p, q, r}$ with $z \in \R$, $\mathcal{C}(\mathcal{S}) = \mathcal{T}(\mathcal{S}) =\{[\theta]\}$ and  for $z \in \C \setminus \R$, $\mathcal{C}(\mathcal{S}^+_{N, p, q, r, z})=\emptyset$ and $\mathcal{T}(\mathcal{S}^+_{N, p, q, r, z})=\{[\theta]\}$, where $\theta$ is the Lee form of Tricerri's metric. 
\end{remark}

From the remark above and \ref{unique} we now derive:

\begin{corollary} The Inoue surfaces have no LCK metric with potential.
\end{corollary}

\begin{proof}
 In \cite[Lemma 3.7]{ad} it is shown that on a compact complex manifold, if $g$ is an LCK metric with potential with the Lee form $\theta$, then for any $t\geq 1$, $t\theta$ is also the Lee form of an LCK metric with potential.  However, by \ref{unique} and the results in \cite{ad}, on the Inoue surfaces, $t\theta$, where $\theta$ is the one-form considered above, cannot be the Lee form of an LCK metric for an $t \in \R \setminus \{1\}$, therefore there is no LCK metric with potential on all $\mathcal{S}^0$, $\mathcal{S}^+, \mathcal{S}^-$. 
 \end{proof}
 
 We recall that for $\mathcal{S}^0$ we proved in \ref{claimm} a stronger result, namely that it cannot admit LCK structures $(g,\omega, \theta)$  $d_{\theta}$-exact two-form $\omega$. By using a similar argument we can prove more:
 
 \begin{proposition} On the surfaces $\mathcal{S}^+_{N, p, q, r, 0}$ and $\mathcal{S}^-$ there exist no LCK metrics which are $d_\theta$-exact. 
 \end{proposition}
 \begin{proof} We shall use the solvmanifold structure of $\mathcal{S}^+$ and $\mathcal{S}^-$. In \cite{has}, $\mathcal{S}^{\pm}$ are described as solvmanifolds $G/\Gamma$, where $G$ is a solvable Lie group with Lie algebra generated by $\{e_1, e_2, e_3, e_4\}$ satisfying:
 
 $$[e_2, e_3]=-e_1, \qquad [e_2, e_4]=-e_2, \qquad [e_3, e_4]=e_3.$$
 The standard complex structure $J$ is $G$-left-invariant and satisifies:
 $$Je_1=e_2, \qquad Je_2=-e_1, \qquad Je_3=e_4 - a e_2, \qquad Je_4=-e_3 - ae_1,$$
 where $a \in \R$.
 The form $\theta$ is the dual of the left-invariant vector field $e_4$.
 
 By contradiction, assume  there exists $\omega=d_{\theta}\eta$ an LCK form on $\mathcal{S}^+$ or $\mathcal{S}^-$. By  \cite [Proposition 1.2]{s},  we may choose a left-invariant form $\eta_0$ such that $d_{\theta}\eta_0$ is still LCK. Then
 $$d_{\theta}\eta_0(e_1, Je_1)=-\eta([e_1, e_2]) - \theta \wedge \eta_0 (e_1, e_2)=0,$$
 contradicting, thus, the fact that $d_{\theta}\eta_0$ is the fundamental form of a Hermitian metric.
 \end{proof}
 
 \begin{remark}
 The  nonexistence of $d_\theta$-exact LCK metrics on some manifolds which cannot admit LCK metrics with potential is related to \cite [Conjecture 1.5]{ovv}, which states that on a compact manifold, a $d_\theta$-exact LCK form is actually with potential. 
 \end{remark}
 
 The same result of nonexistence of LCK metrics which are $d_\theta$-exact holds for Oeljeklaus-Toma manifolds. They are generalizations of Inoue surfaces $\mathcal{S}^0$ and in \cite{kas} it is proven that they are solvmanifolds. Indeed, in \cite[Section 6]{kas}, it is proven that Oeljeklaus-Toma manifolds are isomorphic to quotients $G/\Gamma$, with $G$ a solvable Lie group.
 The Lie algebra $\mathfrak{g}$ has 2s+2 generators
$$\mathfrak{g} = \langle A_1, \ldots, A_s, B_1, \ldots B_s, C_1, C_2\rangle$$ 
and the nonzero structure equations are
\begin{equation*}
\begin{split}
[A_i, B_i] & = B_i \\
 [A_i, C_1]& = -\frac{1}{2}C_1 + \alpha_i C_2\\
 [A_i, C_2]& = -\alpha_i C_1 - \frac{1}{2} C_2
\end{split}
\end{equation*}
with $\alpha_i \in \R$.

Then the left invariant complex structure $J$ is given by $J A_i=B_i$ for $i= 1, 2, \ldots, s$ and $J C_1=C_2$. Let $\mathfrak{g}^* =\langle a_1, a_2, \ldots, a_s, b_1, b_2, \ldots, b_s, c_1, c_2\rangle$. It was proven in \cite{ot} that the first Betti number of Oeljeklaus-Toma manifolds is $s$, hence $H^1=\langle a_1, a_2, \ldots, a_s\rangle$. Therefore, any closed one-form is up to a global exact factor of type $\theta = r_1 a_1 + r_2 a_2 + \ldots r_n a_n$, for some real numbers $r_1, r_2, \ldots r_n$. Since each $a_i$ is invariant, so is $\theta$.

Let us assume that there exists an LCK form $\omega=d_{\theta}\eta$. By \cite [Proposition 1.2]{s}, we may assume that also $\eta$ is left-invariant. Then
$$\omega(C_1, JC_1) = \omega(C_1, C_2) = d\eta (C_1, C_2) - \theta \wedge \eta (C_1, C_2)=0$$
which contradicts the fact that $\omega$ is the fundamental form of a Hermitian metric. Thus, we proved:

\begin{proposition} On Oeljeklaus-Toma manifolds there are no $d_\theta$-exact LCK forms, for any closed one-form $\theta$.  
\end{proposition}

\section{Morse-Novikov cohomology of other LCK surfaces}

We briefly discuss the Morse-Novikov cohomology of other compact complex surfaces which are known to admit LCK metrics. 

Since the blow-up of a manifold at a point is LCK if and only if the manifold itself is LCK (see \cite{vuli}), we are only interested in the minimal model of the surface (\ie  not containing smooth rational curves of self-intersection -1).

LCK metrics have been found in both classes of non-K\"ahler surfaces, $\mathrm{VI}$ and $\mathrm{VII}$ (see \cite{kod}), whose minimal models are denoted by $\mathrm{VI}_0$ and $\mathrm{VII}_0$. They are the only classes of surfaces in which LCK metrics may exist.  

The known examples of LCK surfaces among class $\mathrm{VI}_0$ are properly elliptic surfaces and Kodaira surfaces and are actually Vaisman (see \cite{bel}). Therefore, by \ref{spanioli}, the Morse-Novikov cohomology with respect to the Lee forms of Vaisman metrics vanishes. 

Class $\mathrm{VII}_0$ consists of minimal complex surfaces with $b_1=1$ and Kodaira dimension $-\infty$. It further divides into two subclasses, namely, with $b_2=0$ and $b_2 >0$. In the first case, the classification is complete (see \cite{bog}). They are either Inoue surfaces (for which we computed the Morse-Novikov cohomology in this paper) or Hopf surfaces, for which we can conclude:.
\begin{proposition} 
The Hopf surfaces have vanishing Morse Novikov cohomology with respect to any closed one-form.  
\end{proposition}
\begin{proof}
Recall that the Hopf surfaces are finitely covered by $S^1 \times S^3$ and  the Morse-Novikov cohomology reflects the topology and not  the complex structure of a manifold. Moreover, as $b_1=1$, all closed one-forms are proportional (with a real multiplicative factor), up to an exact one-form, and thus they are parallel with respect to the natural product metric, and hence  \ref{spanioli} yields the vanishing of the twisted cohomology. 
\end{proof}

As regards class $\mathrm{VII}_0$ with $b_2 > 0$, the only known examples are the {\em Kato surfaces}. They were introduced in \cite{km} and in \cite{nak} it was proven that they can be deformed as complex surfaces to a blow-up at finitely many points of the Hopf surface $S^1 \times S^3$. In particular, they are diffeomorphic to $(S^1 \times S^3) \# n \bar{\C\mathbb{P}^1}$, where $n$ is the number of blown-up points. In fact, a stronger result was proved by Nakamura in \cite{nk}, where it is shown that any surface from class $\mathrm{VII}_0$ with a cycle of rational curves is a complex deformation of a blow-up of a Hopf surface. 

By \cite{bru}, all Kato surfaces carry  LCK metrics. Since $b_1=1$, as above, all the closed one-forms are proportional (with a real factor), up to an exact one-form, and identify with the pullback on $(S^1 \times S^3) \# n \bar{\C\mathbb{P}^1}$ of the multiples of the volume form of the circle $S^1$.

It was shown in \cite[Lemma 4.2]{fp} that for any closed one-form $\theta$ on a Kato surface $S$ (and more generally on a surface of class $\mathrm{VII}_0$ with a cycle of rational curves), we have $H^2_{\theta}(S) \simeq \R^{b_{2}(S)}$ and the rest of the Morse-Novikov cohomology groups vanish. 

\begin{remark}
The result in \cite[Lemma 4.2]{fp}  also follows from the following relation proven  in \cite{yz},  between the Morse-Novikov cohomology groups of a compact surface and its blow-up at a point:
\begin{equation*}
\begin{split}
 H^2_{\pi^* \theta} (\tilde{M}_p)&\simeq H^{2}_{\theta}(M) \oplus \R, \\
 H^i_{\pi^* \theta} (\tilde{M}_p)&\simeq H^i_{\theta}(M), \quad i \neq 2.
 \end{split}
 \end{equation*}
where $\pi: \tilde{M}_p \rightarrow M$ is the blow-up of $M$ at the point $p$ (which was, in fact, proven in the more general case of a $n$-dimensional manifold and for the blow-up along a submanifold).
One now takes $M$ to be $S^1 \times S^3$ and $\theta$ any real multiple of $\vartheta$ (which denotes the volume form of the circle). Since for any $i \geq 0$, $H^i_{\theta}(S^1 \times S^3)=0$, one reobtains the cited result.
\end{remark}

\hfill

\noindent{\bf Acknowledgements:} I am greatful to Liviu Ornea for his encouragement and valuable ideas and suggestions that improved this paper and to Massimiliano Pontecorvo for his beautiful explanations and insight  about complex surfaces. Many thanks to Daniele Angella, Nicolina Istrati and Miron Stanciufor very stimulating discussions. I thank Andrei Pajitnov for drawing my attention to the results in the paper \cite{paj}.

I also thank Paolo Piccinni and the University of Rome ``La Sapienza'' for hospitality during part of the work at this paper.

\end{document}